\documentclass[10pt,twoside]{article}
\usepackage{amsfonts,amsmath,amssymb,amsopn,amsthm,mdwlist}
\usepackage[cns,uns]{llconnex}

\newcommand{\Set}{{\bf Set}}

\newcommand{\Vect}{{\bf Vect}}

\newcommand{\too}{\longrightarrow}

\newcommand{\oto}[1]{\stackrel{\underrightarrow{\stackrel{#1}{\,\,\,\,\hphantom{\too}}}}{}}  

\theoremstyle{definition}

\newtheorem{prop}{Proposition}

\author{\sc Brian Day \footnote{Department of Mathematics, Macquarie University, NSW 2109, Australia.} }
\date{September 3, 2009}
\title{Note on the Fusion Map and Hopf Algebras}

\begin{document}

\maketitle

%\vspace{2em}
\begin{abstract}
%\vspace{2em}
We discuss an inversion property of the fusion map associated to many semibialgebras. A brief remark on Tannaka-type reconstruction for
VN-bialgebras has been added.
\end{abstract}

\vspace{1cm}
%\section{Introduction}

Let $\mathcal C = (\mathcal C, \tens, I, c)$ be a symmetric (or just braided) monoidal category. A Von Neumann ``core'' in $\mathcal C$ is firstly
a semibialgebra in $\mathcal C$, that is, an object $A$ in $\mathcal C$ with an associative multiplication: \[ \mu : A \tens A \too A \]
$(\mu_3 = \mu(1 \tens \mu) = \mu(\mu \tens 1) : A \tens A \tens A \too A)$ and a coassociative comultiplication: \[ \delta: A \too A \tens A \]
$(\delta_3 = (1 \tens \delta)\delta = (\delta \tens 1)\delta : A \too A \tens A \tens A )$  such that: \[ \delta\mu = (\mu \tens \mu)(1 \tens c \tens 1)
(\delta \tens \delta) : A \tens A \too A \tens A \] \noindent It is also equipped with an endomorphism \[ S: A \too A \] in $\mathcal C$ such that:
\[ \mu_3(1 \tens S \tens 1)\delta_3 = 1: A \too A \]

The name ``Von Neumann core'' stems partly from the notion of a Von Neumann regular semigroup, which is then precisely a VN-core in $\Set$, while
the free vector space on it is a particular type of VN-core in $\Vect$, and partly from the properties of the paths which generate a (row-finite) 
graph algebra\cite{Raeburn}.

The fusion map \[ f = (1 \tens \mu)(\delta \tens 1) : A \tens A \too A \tens A \] then satisfies the fusion equation by the semibialgebra axiom of
$A$ (see \cite{Street}), and if we set: \[ g = (1 \tens \mu)(1 \tens S \tens 1)(\delta \tens 1) \] as a tentative ``inverse'' to $f$, then we get
the following (partial) results:
\begin{prop} \cite{McCurdy} $fgf = f$ for any VN-core.\end{prop}
\begin{proof}
Define the (left) Fourier transform $l(\alpha)$ of a map $\alpha : A \too B$ to be the composite
\[ A \tens B \oto{\delta \tens 1} A \tens A \tens B \oto{1 \tens \alpha \tens 1} A \tens B \tens B \oto{1 \tens \mu} A \tens B \]
Then $l(\alpha \star \beta) = l(\alpha)l(\beta)$, where $\star$ is the convolution $\alpha \star \beta = \mu(\alpha \tens \beta)\delta$ of two
maps $\alpha$ and $\beta$ from $A$ to $B$.
Thus:\begin{align*}
	fgf 	&= l(1)l(S)l(1) \\
		&= l(1 \star S \star 1) \\
		&= l(1) \\
		&= f
\end{align*} since $1 \star S \star 1 = 1$ by the definition of $VN$-core.
\end{proof}

\begin{prop} $gfg = g$ if $S^2 = 1$ and $S$ is an antihomomorphism either of algebras or of coalgebras.\end{prop} The proof is straightforward.

Recall that a $VN$-core is called ``unital''\cite{DP} if it satisfies the (stronger) axiom 
\[ 1 \tens \eta = (1 \tens \mu)(1 \tens S \tens 1)\delta_3 : A \too A \tens A \]
where $A$ is assumed to have the unit $\eta : I \too A$. (A unital VN-core in $\mathcal C = \Set$ is precisely a group).

\begin{prop}\cite{DP} $gf = 1$ for any unital VN-core.\end{prop}

Note that, in general, if for a map $f$ there exists a map $g$ with $fgf = f$, then we can always find a map $h$ with $fhf = f$ and $hfh = h$ provided
idempotents split in $\mathcal C$.

A semibialgebra is called a very weak bialgebra in \cite{DP} if it also has both a unit $\eta : I \too A$ ($\mu(1 \tens \eta) = \mu(\eta \tens 1) = 1$)
and a counit $\epsilon : A \too I$ ($(1 \tens \epsilon)\delta = (\epsilon \tens 1)\delta = 1$). A very weak bialgebra $A$ is then called a very
weak Hopf algebra if it is equipped with a map $S: A \too A$ satisfying the axioms:
\begin{align*}
	\mu(S \tens 1)\delta &= t := (1 \tens \epsilon \mu)(c \tens 1)(1 \tens \delta\eta) \\
	\mu(1 \tens S)\delta &= r := (\epsilon \mu \tens 1)(1 \tens c^{-1})(\delta \eta \tens 1) \\
	\mu_3(S \tens 1 \tens S)\delta_3 &= S
\end{align*}

Hence $S \star 1 \star S = S$ so that $gfg = g$ and, as a consequence of the semibialgebra axiom, we have $1 \star t = 1$ (see \cite{PS}) whence
$1 \star S \star 1 = 1$ so that $fgf = f$ (using $S \star 1 = t$ by the first axiom).

{\sc Example}: Suppose that $(A,\mu,\delta,\eta,\epsilon)$ is a bialgebra for which $\delta$ is \emph{not} known to be coassociative, and
suppose that $A$ is also equipped with a map $S: A\too A$ (not necessarily an antihomomorphism), and invertible elements $\lambda : I \too A$ and
$\rho : I \too A$ such that the standard Drinfel'd axioms hold, namely: \begin{align*}
	\mu_3(S \tens \lambda \tens 1)\delta &= \lambda\epsilon \\
	\mu_3(1 \tens \rho \tens S)\delta &= \rho\epsilon
\end{align*}

\begin{prop} This is a quasi-VN-bialgebra in the sense that both the equations \[ \mu_3(1 \tens S \tens 1)(\delta \tens 1)\delta = 1 \]
and \[ \mu_3(1 \tens S \tens 1)(1 \tens \delta)\delta = 1 \] hold.\end{prop}

%Then at least \[ (1 \star S) \star 1 = 1 \star (S \star 1) \]
%however here \[ l(\alpha \star 1) \neq l(\alpha)l(1) \] and \[ l(1 \star \beta) \neq l(1) l(\beta) \]
%in general.

Note that the two standard Drinfel'd conditions were still satisfied in the definition of a weak quasi-Hopf algebra (in the sense of Haring-Oldenburg et
al. \cite{HO}).

We note also that any VN-core $A$ in $\Vect_k$ can be completed to a (VN-) bialgebra $A \oplus k$ in a fairly obvious way.

Moreover, there are Tannaka-type reconstruction results for VN-bialgebras, based on the notion of a partially compact monoidal category,
which is simply a $k$-linear monoidal category $(\mathcal A,\tens,I)$ equipped with an antipode functor $S$ and two $k$-linear natural
transformations:

\[ \mathcal A(A \tens B, C) \oto{n} \mathcal A(A,C \tens SB) \oto{e} \mathcal A(A \tens B, C) \]

\noindent such that $ene = e$ and $nen = n$. The reconstruction results use the facts that the finite-dimensional left $k$-representations of a VN-bialgebra
$(A,\mu,\delta,\eta,\epsilon,S)$ satisfying $S(xy)= SySx$ and $S1 = 1$ form such a partially compact monoidal category, and, in general, the
self-dual representations of any VN-bialgebra form a partially compact monoidal category.

Enquires (etc.) regarding this article can be made to the author through Micah McCurdy (Macquarie University), who kindly typed the manuscript.

\end{document}